%% file: main.tex
\PassOptionsToPackage{rgb}{xcolor}

\documentclass[11pt,a4paper]{article}  
\usepackage{fourier} 
\usepackage{amsmath,amsbsy,bm,amssymb,enumerate,mathrsfs}
\usepackage[utf8]{inputenc} 
\usepackage{utf8math}
\usepackage{todonotes}
\usepackage{mdframed}
\usepackage{anysize}
\usepackage{graphicx}

\usepackage{tikz}
\input{ipestyle} 

\usepackage{tikz} 

\newcommand{\R}{\mathbb R}

\newcommand{\Z}{\mathbb Z}

\renewcommand{\epsilon}{\varepsilon} 

\newcommand{\cone}{\operatorname{cone}}
\newcommand{\intcone}{\operatorname{intcone}}

\usepackage{etoolbox}
\makeatletter
\patchcmd{\@sect}{#8}{\boldmath #8}{}{}
\let\ori@chapter\@chapter
\def\@chapter[#1]#2{\ori@chapter[\boldmath#1]{\boldmath#2}}
\makeatother

\newcommand{\C}{\mathcal{C}}
\newcommand{\V}{\mathcal{V}}

\DeclareMathOperator{\rank}{rank}

\usepackage[T1]{fontenc}
\usepackage{graphicx}
\usepackage{amsthm}
\theoremstyle{plain}
\newtheorem{theorem}{Theorem}

\newtheorem{lemma}[theorem]{Lemma}

\theoremstyle{definition}

\newtheorem{remark}{Remark}

\title{Forall-exist statements in pseudopolynomial time} 

\author{
Eleonore Bach 
 \thanks{ EPFL, Switzerland, 
  \texttt{eleonre.bach@epfl.ch}}
 \and 
  Friedrich Eisenbrand 
\thanks{ EPFL, 
 Switzerland, 
 { \texttt{friedrich.eisenbrand@epfl.ch}}}
 \and
 Thomas Rothvoss \thanks{
 University of Washington, USA, 
  \texttt{rothvoss@uw.edu}}
\and
 Robert Weismantel \thanks{
 ETH Zürich, 
 Switzerland, 
  \texttt{robert.weismantel@ifor.math.ethz.ch}}
}
\date{\today}
\begin{document}

\maketitle 

\begin{abstract}
  \noindent 
  Given a convex set $Q ⊆ ℝ^m$ and an integer matrix $W ∈ ℤ^{m ×n}$, we  consider statements of the form
  $ ∀ b ∈ Q ∩ ℤ^m$ $∃ x  ∈ ℤ^n$  \text{s.t.} $Wx ≤ b$.  
  Such statements can be verified in polynomial time with the
  algorithm of Kannan and its improvements if $n$ is fixed and $Q$ is
  a polyhedron. The running time of the best-known algorithms is
  doubly exponential in~$n$.

  We provide a
  pseudopolynomial-time algorithm if $m$ is fixed. Its running time is
  $(m Δ)^{O(m^2)}$ where $Δ$ is the largest absolute value of an
  entry in $W$.  Furthermore it applies to general convex sets
  $Q$. 
 
\end{abstract}

\section{Introduction}
\label{sec:introduction}

An \emph{integer linear program (ILP)} is a discrete optimization problem of the following kind 
\begin{equation}
\label{eq:1}
  \max \big\{ c^T x : Ax = b, \, x ≥ \bm{0}, \, x ∈ ℤ^n \big\}  
\end{equation} 
where $A ∈ ℤ^{m×n}$, $b ∈ℤ^m$ and $c ∈ ℤ^n$.
Many algorithmic problems 
can be modeled and solved  as an integer program. Integer programming is a showcase of progress and development in the field of \emph{algorithms and complexity}. 
If the number of variables in~\eqref{eq:1} is fixed, then Lenstra-type  algorithms~\cite{lenstra1983integer,kannan1987minkowski} solve integer programming in polynomial time. The recent result of Reis and Rothvoss~\cite{reis2023subspace}, together with an algorithm of Dadush~\cite{dadush2012integer}  has a running time of $(\log n )^{O(n)}$ times a polynomial in the binary encoding-length of the input.

\noindent
In 1982 Papadimitriou~\cite{papadimitriou1981complexity} has shown that such integer programs in standard form 
can be solved in pseudopolynomial time, if the number $m$ of rows of $A ∈ℤ^{m×n}$ is fixed. The running time of Papadimitriou's algorithm is $(m Δ)^{O(m^2)}$, where $Δ$ is the largest absolute value of an entry of $A$. 
Papadimitriou's algorithms was recently improved.
Standard form IPs can be solved in time $(mΔ)^{O(m)}$, see~\cite{eisenbrand2019proximity,jansen2023integer}. Knop, Pilipczuk, and Wrochna~\cite{knop2020tight} showed that this running time is optimal up to constants in the exponent. This lower bound is assuming the  exponential-time-hypothesis~\cite{impagliazzo2001problems}. In presence of upper bounds on the variables, the best-known pseudopolynomial-time algorithms~\cite{eisenbrand2019proximity} still have a complexity of $(mΔ)^{O(m^2)}$. Whether this running time is optimal, is a highly visible open problem.


\medskip \noindent 
Central to this paper are  \emph{forall-exist statements} of the form
\begin{equation}
  \label{eq:26}
     ∀ b ∈ Q ∩ ℤ^m \, ∃ x  ∈ ℤ^n \quad   \text{s.t.}\quad  Wx ≤ b, 
   \end{equation}
   where  $Q ⊆ ℝ^m$ is a given convex set and  $W ∈ ℤ^{m ×n}$ is a given integer matrix. 
   
Forall-exist statements are a substantial generalization of integer programming. For a given right-hand-side $b∈ℤ^m$ and $Q = \{b\}$ deciding correctness of the statement~\eqref{eq:26} is an integer feasibility problem. 
It comes as no surprise  that problem~\eqref{eq:26}  belongs to the second level of the polynomial hierarchy and is $\Pi_2$-complete~\cite{stockmeyer1976polynomial,wrathall1976complete}.
Kannan~\cite{kannan1992lattice} provided an algorithm to decide forall-exist statements that runs in polynomial time if the dimension $n$ (number of columns of $W$) and $m$ are fixed. Eisenbrand and Shmonin~\cite{eisenbrand2008parametric} extended this result to the case where only $n$ is assumed to be a constant.

  Forall-exist statements are of interest in several
  scientific disciplines. A classical example from number theory is the \emph{Frobenius
    problem}~\cite{kannan1992lattice}. Recently, forall-exist statements are of increasing  importance  in the field of \emph{fixed-parameter complexity}  see, e.g.~\cite{gavenvciak2022integer,knop2018unifying}. 
  A nice application is in the scope of \emph{fair allocations}~\cite{bredereck2020high,crampton2017parameterized}. 

\subsubsection*{Contributions} 
\label{sec:contr-this-paper}

Our main result is a pseudopolynomial time algorithm to decide forall-exist statements in the case where the number $m$ of rows of the matrix $W ∈ℤ^{m ×n}$  is fixed.  More precisely, the novel  contributions of this paper are the following. 
\begin{enumerate}[i)]
\item  We show that a decision problem~\eqref{eq:26} can be decided in time $(m Δ)^{O(m^2)}$. Here  $Δ $ is the largest absolute value of a component of $W$.  In case that the answer is negative, our algorithm provides a $b \in Q \cap \mathbb{Z}^m$ so that the system $Wx \leq b, \; x \in \mathbb{Z}^n$ is infeasible.
\end{enumerate}
This result is  via a sequence of reductions that leads  to a conjunction of simpler forall-exist statements, for which the domain of the $∃$-quantifyer is a finite set of integer vectors. The number of such sub-problems itself is
\begin{displaymath}
  \binom{n}{m} ⋅(m Δ)^{O(m)} = (m Δ)^{O(m^2)}.
\end{displaymath}
The last equality follows form the fact that  we can assume that $W$ does not have repeated columns and hence $n ≤ (2 Δ+1)^m$. 

This running time is not higher than state-of-the-art algorithms for integer programming with lower and upper bounds on its variables~\cite{eisenbrand2019proximity} in the pseudopolynomial-time regime where $m$ is fixed. In particular, the algorithm presented here does not show double exponential dependence on the number of variables.
The \emph{ETH}-based lower bound of Knop et al.~\cite{knop2020tight} of  $(m Δ)^{Ω(m)}$ for integer programming problems~\eqref{eq:1} transfers to the same lower bound for forall-exist problems~\eqref{eq:26}, by setting $Q = \{b\}$, the right-hand-side of~\eqref{eq:1}. 


  \begin{enumerate}[i)] 
    \setcounter{enumi}{1}
  \item A novel feature of our algorithm is that it applies to general convex sets $Q ⊆ℝ^m$, 
    whereas Kannan's algorithm is described and analyzed for polyhedra
    only.  
  \end{enumerate}
  %
  The analysis of algorithms involving  a convex set $Q$ requires a fair amount of technical care, see, e.g.~\cite{grotschel2012geometric}. 
  We need to be able to solve the following problems involving $Q$. Our algorithm generates  rational polyhedra $P ⊆ ℝ^m$ for which it needs to decide whether $Q ∩ P$ contains an integer point, or for a given $x^* ∈ P$ it has to decide membership in $Q$. For the latter task, it is enough to have access to $Q$ in form of a \emph{membership oracle}~\cite{grotschel2012geometric}. A query to this oracle has cost $1$.  The former task is more subtle. Using the state of the art integer programming algorithm~\cite{reis2023subspace} this question can be decided in time $(\log(m)^{O(m)}$ times a polynomial in $\log(R)$ where $R>1$ is the radius of a ball containing $Q$. 
  We abstract from such a detailed running time analysis  by accounting cost~${1}$ for this task as well. 
   
  We  also provide new structural results on specific forall-exist problems that have attracted recent attention~\cite{cslovjecsek2024parameterized,aliev2010feasibility}. The \emph{diagonal Frobenius number} of a pointed cone $\cone(W) = \{ Wx ：x∈ ℝ_{\geq 0}^n\}$  where $W ∈ ℤ^{m ×n}$,  is the smallest $t^*≥0$ such that one has the following:  For all $c ∈ \cone(W)∩ ℤ^m $  that are conic combinations derived  with weights more than $t^*$ in every generator one has that these points are \emph{integer conic combinations} as well.  

\begin{enumerate}[i)]
  \setcounter{enumi}{2}
\item  We show a bound on the diagonal Frobenius number of $(m Δ)^{O(m)}$ which yields an improvement of  the previous-best bound of Aliev and Henk~\cite{aliev2010feasibility} in our parameter setting. 
\end{enumerate}

\subsubsection*{Comparison with the polynomial-time algorithm in fixed dimension}

The breakthrough of Kannan~\cite{kannan1992lattice} and its subsequent improvements~\cite{eisenbrand2008parametric}  is a  polynomial time algorithm  if the dimension $n$ (number of columns of $W$) is fixed. 
  The running time of these algorithms is doubly exponential in the number
  of variables $n$. More precisely, these algorithms require a running
  time of at least
  \begin{equation}
    \label{eq:14}
    (m \log Δ)^{2^n}.
  \end{equation}
To the best of our knowledge, this is the only algorithm with a nontrivial analysis of its running time that is available for tackling forall-exist statements. 

 By ignoring the dependence on the binary encoding-length of $Δ$ and dropping constants, the  achieved  running time~\eqref{eq:14}  of  Kannan's algorithm~\cite{kannan1990test} can be lower-bounded by $Ω(m ^{2^n})$. Then, up to constant factors, one can see that our algorithm is more efficient in the parameter-range
\begin{equation}
  \label{eq:2}
    m^2 \log (m Δ) ≤ 2^n \log (m). 
  \end{equation}
  The number $m$  (rows of $W$) can in principle be exponential in the number of variables $n$. This is a setting, where Kannan's algorithm is more efficient than our pseudopolynomial-time algorithm. Another interesting setting is when $m ≤ n^k$ for some constant $k$.  This applies,  for example in the context of fair allocation~\cite{bredereck2020high}.  To illustrate the efficiency of our algorithm in this case, we can assume that $Δ$ is at least $m$. If the  left-hand-side of~\eqref{eq:2}  exceeds the right-hand-side, then 
  \begin{displaymath}    
    2 n^{2k} \log (Δ) > 2^n \quad \iff \quad  \log (Δ) > 2^{ n - 2 k \log n -1}. 
  \end{displaymath}
Since $k$ is a constant, this means that $Δ$ has to be {\bf doubly-exponential} in $n$. In other words, the number of bits to encode the largest entry of  $W$ has to be \emph{exponential} in~$n$.   Outside of this regime and under the assumption that  $m$ is polynomial  in $n$, the algorithm proposed here is more efficient in terms of worst-case running time. 

\section{A birds-eye perspective on our approach} 
\label{sec:gener-forall-stat}

Our main result is via a sequence of reductions.  The details of this reduction are explained in Section~\ref{sec:sequence-reductions}. We start here by recalling the starting point and then describe the final problem  in this sequence and its solution, thereby providing an overview as well as a first algorithmic result. 
Throughout $Q ⊆ ℝ^m$ denotes a convex set and $W ∈ℤ^{m ×n}$  denotes an integer  matrix with  $\|W\|_∞ ≤ Δ$. We are concerned with the following decision problem. 
\begin{mdframed} 
  Given $Q ⊆ ℝ^m$ and $W ∈ℤ^{m ×n}$, decide whether
  \begin{equation}\label{eq:12}
    ∀ b ∈Q ∩ ℤ^m \quad \text{there exists} \quad  x ∈ ℤ^n \quad  \text{with} \quad  Wx ≤ b. 
\end{equation}
\end{mdframed}
Our main result is  a  reduction of problem~\eqref{eq:12} to $\binom{n}{m} ⋅ (m ⋅Δ)^{O(m)}$ many simpler forall-exist problems of the following kind. 
\begin{mdframed} 
  Given a convex set $Q ⊆ ℝ^m$,  and a \emph{finite} set ${\cal C} ⊆ ℤ^m$ with  $\|\C\|_∞ ≤ (m ⋅Δ)^{O(m)}$. Decide the validity of the statement 
 \begin{equation}
\label{eq:3}
   ∀ b ∈Q ∩ ℤ^m \, ∃ c ∈{\cal C}：   c≤ b. 
 \end{equation}
\end{mdframed}
Here $\|\C\|_∞$ denotes the largest infinity norm of an element in $\C$. 
Notice that in contrast to the forall-exist statement of our departure, the domain of the variable in the scope of the $∃$-quantifier at the end-of our reduction  is  \emph{finite}. In fact $|\C| ≤ (mΔ)^{O(m^2)}$ follows from a counting argument. 
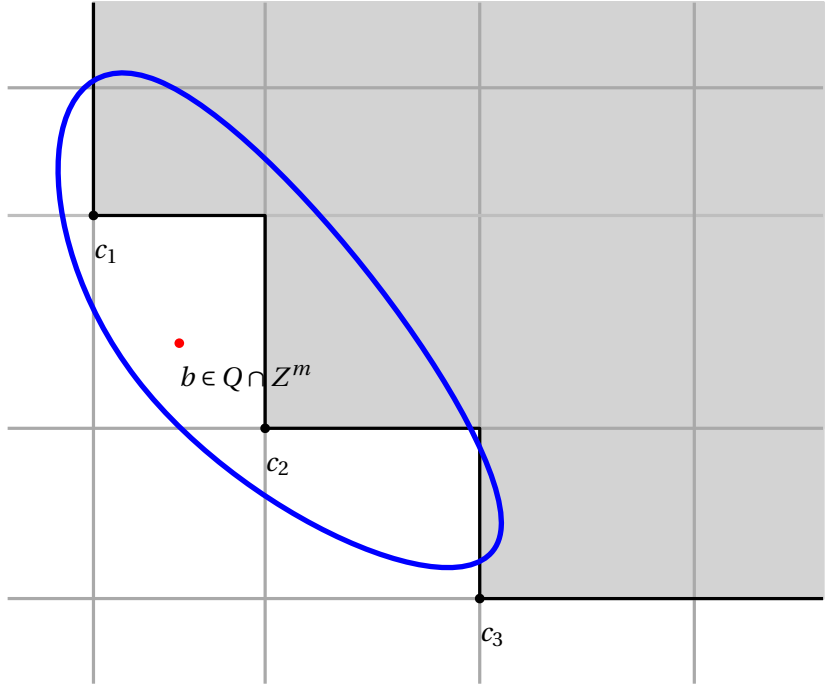
\begin{figure}[h]
\centering 
\input{arrangement-new}   
      \caption{A schematic picture of problem~\eqref{eq:3}. The set $Q$ is drawn in blue. The elements of the  set $\C$ are $c_1,c_2$ and $c_3$. The area in grey corresponds to all points $x ∈ ℝ^m$ such that there exists a $c ∈ \C$ with $c ≤x$. The point $b$ in red is an integral point in $Q$ that is not contained in the grey area and hence is a counter-example of the validity of the corresponding forall-exist statement.}
    \label{fig:1}
  \end{figure}
  The running time that is necessary to generate $\binom{n}{m} ⋅ (m ⋅Δ)^{O(m)}$ many simpler forall-exist problems will be $\binom{n}{m} ⋅ (m ⋅Δ)^{O(m)}$  as well. Figure~\ref{fig:1} illustrates the  exist statement~\eqref{eq:3}.

 We conclude here by showing that~\eqref{eq:3} can be solved in time $(m ⋅Δ)^{O(m^2)}$.

\begin{theorem}
  \label{thr:1}
    A forall-exist statement~\eqref{eq:3}  can be decided in time $(m ⋅Δ)^{O(m^2)}$.  
\end{theorem}

\begin{proof}
  The goal is to find a \emph{counter-example}, i.e., an integer point $b ∈ Q ∩ℤ^m$ such that for every $c ∈ {\cal C}$, there exists an index $i ∈ \{1,\dots,m\}$ such that $b_i < c_i$. Since all numbers are integers, the latter condition is equivalent to
  $b_i ≤ c_i-(1/2)$. We now consider the hyperplane arrangement defined by the axis-parallel hyperplanes
  \begin{equation}
    \label{eq:4}
    H_c^i = \left\{ x ∈ ℝ^m ： x_i = c_i-(1/2) , \quad c ∈ {\cal C}, i ∈ \{1,\dots,m\} \right\}. 
  \end{equation}
  This partitions $ℝ^m$ into finite and infinite cells. Let us describe these cells precisely. For every component $i ∈ \{1,\dots,m\}$, let $\{ c_i ： c ∈ \C\}$ be the set of $i$-th components of elements  of $\C$. Let $ℓ_i^1 < \dots < ℓ_i^{k_i}$ be an ordering of this set. A cell $\V$ is then determined by a tuple
  \begin{displaymath}
    (j_1,\dots,j_m) ∈ \{0,\dots,k_1\} × \cdots \times \{0,\dots,k_m\} 
  \end{displaymath}
  and it has the form 
  \begin{displaymath}
    \V = \big\{ x ∈ ℝ^m ： \ell_i^{j_1} - (1/2) ≤x_i ≤\ell_i^{j_1+1} - (1/2) \big\}, 
  \end{displaymath}
  where $\ell_i^{0} = -∞$ and $\ell_i^{k_i+1} = +∞$. 
  A potential counter-example must lie in the interior of a cell $\V$, since it is integral. Furthermore, the interior of $\V$ either is fully contained in the union of the cones 
  \begin{equation}
    \label{eq:6}
\displaystyle     \bigcup _{c ∈ \C} \left( (c - \frac{1}{2} \mathbf{1})  + \mathbb{R}^m_{≥0} \right),
\end{equation}
or it is disjoint from this possibly non-convex set. To find a counterexample, we iterate over all cells $\V$. One iteration is as follows.
\begin{enumerate}[A)] 
\item We check whether
  the interior of $\V$ is contained in the union~\eqref{eq:6}.  This is the case if and only if an arbitrary point from its interior is contained in one of the cones. \label{item:1}
\item In the case in which the interior is not contained in one of the cones, we check whether the integer program
  \begin{equation}
    \label{eq:8}
    \V ∩ Q ∩ ℤ^m 
  \end{equation}
  is feasible. If this is true, a counterexample has been detected and we can stop the process. 
\end{enumerate}
The integer program~\eqref{eq:8} can be solved in time $(\log m  )^{O(m)}$~\cite{reis2023subspace} which is dominated by our final running time.
  It remains to be shown that the number of cells is bounded by $(mΔ)^{O(m^2)}$. Clearly, the number of cells is equal to
  \begin{displaymath}
    \prod_{i=1}^m ({k_i} +1).  
  \end{displaymath}
  Since the infinity norm of each $c ∈ \C$ is bounded by $(mΔ)^{O(m)}$, one has $k_i ≤(mΔ)^{O(m)}$ and therefore, the number of cells is bounded by $(mΔ)^{O(m^2)}$. 
\end{proof}

\section{The sequence of reductions}
\label{sec:sequence-reductions}

The goal of this section is to provide a proof of the following assertion.
\begin{theorem}
  \label{thr:3}
  There exists an algorithm that transforms a forall-exist statement~\eqref{eq:26} into an equivalent conjunction of
  \begin{equation}
    \label{eq:9}
    \binom{n}{m} ⋅
   (mΔ)^{O(m)} 
  \end{equation}
  many forall-exist statements~\eqref{eq:7}. The running time of the algorithm is bounded by  $(mΔ)^{O(m^2)}$.  
\end{theorem}
\begin{remark}
  \label{rem:1}
  The running time of the algorithm of $(mΔ)^{O(m^2)}$ is potentially higher than the number of problems~\eqref{eq:7} in the conjunction. In short, this is because we explicitly enumerate the set $\C$. The bound on the infinity norm of $(mΔ)^{O(m)}$ for each element of $\C$ yields a straight-forward bound of $|\C| =  (mΔ)^{O(m^2)}$. 
\end{remark}


We start with a standard transformation that is more convenient for us, as we use the concepts of a finitely generated cone and of a finitely generated integer cone.  The \emph{cone} generated by the column vectors of a matrix  $C ∈ℤ^{m ×n}$ is the set  $\cone(C) = \{C x : x ∈ \R_{\geq 0}^n\}$.  The \emph{integer cone}  $\intcone(C)$ is defined as  $\intcone(C) = \{  C x ： x ∈ \Z_{\geq 0}^n\}$. The following is a very important key concept. If $C ∈ ℤ^{m ×m}$ is non-singular, then
\begin{equation}
  \label{eq:10}
  \intcone(C) = \cone(C) ∩ Λ(C). 
\end{equation}
Here $Λ(C) = \{ Cx ： x ∈ ℤ^m\}$ is the \emph{(full-dimensional)  lattice} generated by $C$. The matrix $C$ is called \emph{basis} of $Λ(C)$.

We re-write the condition $x ∈ℤ^n$,  $Wx ≤ b$ as $x' ∈ℤ_{\geq 0}^{n'}$,  $W'x' = b$. 
In this way, the latter condition $x' ∈ℤ_{\geq 0}^{n'}$,  $W'x' = b$ can be written as 
$b ∈ \intcone({W}')$. Notice that $\| {W}' \|_∞ = \| {W} \|_∞$ and that $\rank(W') = m$.  We can thus assume, without loss of generality, that our forall-exist statement is  as follows. 
\begin{mdframed} 
  Given $Q ⊆ ℝ^m$,  $W ∈ℤ^{m ×n}$ of rank $m$, decide whether
 \begin{equation}
  \label{eq:7}
   ∀ b ∈Q ∩ ℤ^m：  b ∈ \intcone(W). 
 \end{equation}
\end{mdframed}

\subsection{Enforcing $\mathbf{Q ⊆ \cone(W)}$}
\label{sec:first-simpl}
 
Suppose that there exists an element in $(Q \setminus \cone(W)) ∩ ℤ^m$. Then this element is a counterexample to~\eqref{eq:7}. 
We begin by excluding such counterexamples that are outside of $\cone(W)$ 
by preprocessing via 
integer programming techniques.

 More precisely, this is done by solving integer feasibility problems 
 \begin{displaymath}
   Q ∩ \Big\{ x ∈ ℤ^m ： a^T x ≥ 1\Big\} ≠ \emptyset 
 \end{displaymath}
 for each integral facet-defining inequality $a^T x ≤ 0$ of
 $\cone(W) ⊆ ℝ^m$. As described in the introduction, we account for a running time of $1$ for this test. If $Q$ was explicitly given as a rational polyhedron, then this test can be carried out in time $(\log m)^{O(m)}$ times a polynomial in $\log Δ$ and the binary encoding length of the description of $Q$. Apart from the latter factor, this is dominated by our running time. 

The number of facets is bounded by $\binom{n}{m}$ and the facets can be enumerated in this time-bound as well, see, e.g.~\cite{schrijver1998theory}. From now on, we can assume that $Q ⊆ \cone(W)$.

 \subsection{Reduction to simplicial cones}
 

 \emph{Carathéodory's theorem}, see, e.g.~\cite{schrijver1998theory}  guarantees that each
 $b$ in $\cone(W)$ is contained in $\cone(W_B)$ for a basis
 $B ⊆ \{1,\dots,n\}$ of $W$. Here a \emph{basis} $W_B$ of $W$ is a selection  of $m$ linearly independent columns of $W$. 
 Clearly, the forall-exist statement~\eqref{eq:7} over $Q⊆\cone(W)$ holds, if and only if it holds over all sets $Q ∩ \cone(W_B)$. The number of sets  $Q ∩ \cone(W_B)$ is bounded by $\binom{n}{m}$. 

 Our next lemma shows that, in each of these statements over $Q ∩ \cone(W_B)$,  we can \emph{almost} replace the condition $x ∈ \intcone(W)$ by $x ∈ \intcone(W_B)$.
 \begin{lemma}
   \label{lem:2}
   Let $b ∈ \cone(W_B) ∩ ℤ^m$, then  $b ∈ \intcone(W)$ if and only if there exists an 
 element $v ∈ \intcone(W)$ of norm $\|v\|_∞ \leq (m Δ) ^{O(m)}$ such that
 $b - v ∈ \intcone(W_B)$.
\end{lemma}
\noindent
The proof relies on the following theorem. 
\begin{theorem}[Theorem 3.3 in~\cite{eisenbrand2019proximity}]
  \label{thm:prox}
  Consider a feasible integer program of the form
\begin{equation}
\label{eq:ProximityIP}
  \max \big\{ c^T x : Ax = b, \, x ≥ \bm{0}, \, x ∈ ℤ^n \big\}  
\end{equation} 
where $A ∈ ℤ^{m×n}$, $b ∈ℤ^m$ and $c ∈ ℤ^n$ with $\|A\|_∞ ≤Δ$.  
    Let $x^*∈ℝ^n_{\geq 0}$ be an optimal fractional vertex solution of the linear programming relaxation. There exists an optimal solution $z^*∈ℤ^n_{\geq 0}$ of the integer program~\eqref{eq:ProximityIP} such that $\|z^* - x^* \|_1 \leq m(2m \Delta + 1)^m $. 
\end{theorem}

\begin{proof}[Proof of Lemma~\ref{lem:2}]
  If there exists an element $v ∈ \intcone(W)$ with  $b - v ∈ \intcone(W_B)$, then
  \begin{displaymath}
    b ∈ v + \intcone(W_B) ⊆ \intcone(W). 
  \end{displaymath}
  Conversely, let $b ∈ \cone(W_B)$. Then  $b ∈ \intcone(W)$ is equivalent to   the fact that  the following integer program is feasible  
  \begin{equation}
    \label{eq:15}
    \max\Big\{ \bm{0}^T x  ： Wx = b, \, x ∈ ℤ^n_{\geq 0}\Big\}. 
  \end{equation}
  Since $b ∈ \cone(W_B)$, there exists an optimal vertex solution $x^*∈ ℝ_{\geq 0}^n$ of the linear programming relaxation of~\eqref{eq:15} has positive entries only in components $i ∈B$. With Theorem~\ref{thm:prox} it follows that there exists an integer (optimal) solution $z^*∈ℤ^n_{\geq 0}$ with
  \begin{displaymath}
     \|z^* - x^* \|_1 \leq m(2m \Delta + 1)^m .
   \end{displaymath}
   This means that $b$ can be decomposed as $b = u + v$, where $u ∈ \intcone(W_B)$ and $v∈ \intcone(W_{NB})$ such that
   \begin{displaymath}
     v = W_{NB} ⋅z^*_{NB}. 
   \end{displaymath}
   Here we rely on usual notation: $W_{NB}$ is the matrix composed by the columns of $W$ that are indexed by $NB =\overline{B}$ and $z^*_{NB}$ is analogously composed of $z^*$. Notice that $\|z^*_{NB}\|_1 =\|z^*_{NB} - x^*_{NB}\|_1 \leq (m Δ)^{O(m)}$.  Therefore
   \begin{displaymath}
      \| v \|_∞ = \|W_{NB} ⋅z^*_{NB}\|_∞ \leq (mΔ)^{O(m)}.    
   \end{displaymath}
 \end{proof}

 \noindent 
 Let us define $𝒞⊆\intcone(W)$ as the set of all
 elements of $\intcone(W)$ of infinity norm bounded by
 $(mΔ)^{O(m)}$. Notice that this set has cardinality $(mΔ)^{O(m^2)}$. 
The above discussion shows that the statement~\eqref{eq:7} is equivalent to the conjunction over $\binom{n}{m} ⋅(mΔ)^{O(m)}$  statements of the following form. 
\begin{mdframed}
  Given $W ∈ℤ^{m ×m}$ of rank $m$,  $Q ⊆ \cone(W)$  convex and   $𝒞 ⊆ ℤ^m$, where  $\|\C\|_∞ ≤  (mΔ)^{O(m)}$,  decide whether 
\begin{equation}
  \label{eq:13}
   ∀ b ∈Q ∩ ℤ^m \quad ∃ c ∈ 𝒞 \quad \text{such that} \quad  b-c ∈ \intcone(W)
 \end{equation}
 
\end{mdframed}

\subsection{Partitioning in residue classes of $Λ(W)$}
\label{sec:intcone-vs.-cone}

Recall that $Λ(W)  = \{ Wx ： x∈ ℤ^m\}$ is the lattice generated by the non-singular and integral matrix $W ∈ ℤ^{m  ×m}$. The \emph{fundamental parallelepiped} $Π(W)$ is the set
\begin{displaymath}
  Π(W) := \{ Wλ ： λ∈ [0,1)^m \}. 
\end{displaymath}
The volume of  $Π(W)$ is equal to $|\det(W)|$ and corresponds to the number of integer points in $Π(W)$. The Hadamard inequality shows that $|\det(W)| ≤ (mΔ)^{O(m)}$.
The lattice $ℤ^m$ can be partitioned into  residue classes modulo  $Λ(W)$
\begin{displaymath}
  ℤ^m   = \bigcup_{ p ∈ Π(W) ∩ ℤ^m} (p + Λ(W)). 
\end{displaymath}
See, e.g.~\cite{barvinok2002course} for further details.   
Furthermore, we can assume that each element of $Π(W)$ has infinity norm bounded by $m ⋅ Δ$.  This shows that the decision problem~\eqref{eq:13} can be reduced to the conjunction of $(m Δ)^{O(m)}$ decision problems of the following kind, each parameterized by a representative $p ∈ Π(W)∩ℤ^m$ in the fundamental parallelepiped.  
\begin{mdframed}
  Given $W ∈ℤ^{m ×m}$ of rank $m$, $p ∈ Π(W)∩ℤ^m$, $Q ⊆ \cone(W)$  convex and   $𝒞 ⊆ ℤ^m$, decide whether 
  \begin{equation}
    \label{eq:16} 
   ∀ b ∈ Q  ∩ (Λ(W) + p)  \quad ∃ c ∈ 𝒞  \quad \text{such that} \quad  b-c ∈ \intcone(W). 
 \end{equation} 
\end{mdframed}
The number of decision problems of the form~\eqref{eq:16} to which~\eqref{eq:7} reduces to is $(m Δ)^{O(m^2)}$ and the running time involved to arrive at these sub-problems is in the same order of magnitude.  This decision problem is now the point of departure of the final reduction step. 
  
\subsection{Transforming to $ℝ^m_{≥0}$}
\label{sec:transforming-m_0}
Our task is to solve the decision problem~\eqref{eq:16}. We begin by recalling  that $\intcone(W) = Λ(W) ∩ \cone(W) $. Hence  if   $b ∈ Q  ∩ (Λ(W) + p)$ and $c ∈𝒞$ with $b-c ∈ \intcone(W)$ one necessarily has 
\begin{equation}
  \label{eq:17}
    c ≡p \pmod{Λ(W)}.
\end{equation}
Therefore, we can delete from $\C$ all elements for which~\eqref{eq:17} does not hold  
and  we can  re-write the decision problem~\eqref{eq:16} as follows.
\begin{mdframed}
  Given $W ∈ℤ^{m ×m}$ of rank $m$, $p ∈ Π(W)∩ℤ^m$, $Q ⊆ \cone(W)$  convex and   $𝒞 ⊆ Λ(W) + p$, decide whether 
  \begin{equation}
    \label{eq:cone}
   ∀ b ∈ Q  ∩ (Λ(W) + p)  \quad ∃ c ∈ 𝒞  \quad \text{such that} \quad  b  ∈ \cone(W) + c. 
 \end{equation} 
\end{mdframed}
By subtracting $p$ from $Q$ as well as from $\C$, the statement~\eqref{eq:cone} is equivalent to the following.
\begin{equation}
  \label{eq:11}
   ∀ b ∈ Q'   ∩ Λ(W)   \quad ∃ c ∈ \C'  \quad \text{such that} \quad  b-c  ∈ \cone(W), 
\end{equation}
where $Q' = Q-p$ and $\C' = \C-p$. Observe that $\C' ⊆ Λ(W)$. One has
\begin{displaymath}
  Q'   ∩ Λ(W)  =   W ( W^{-1}Q' ∩ ℤ^m),  \quad  \C' = W (W^{-1} \C') \quad \text{and}\quad  \cone(W) = W\, ℝ^m_{≥0}. 
\end{displaymath}
Furthermore, we have  $W^{-1} \C' ⊆ℤ^m$. Recall that $\|\C'\|_∞ ≤ (m Δ)^{O(m)}$. The Hadamard inequality implies that the absolute value of each component of $W^{-1}$ is bounded by $(mΔ)^{O(m)}$. Thus
\begin{displaymath}
  \|W^{-1}\C'\|_∞ ≤ (m Δ)^{O(m)}.
\end{displaymath}
By re-defining $Q$ as $W^{-1}Q'$, $\C$ as $W^{-1} \C' ⊆ℤ^m$ we arrive at the desired simple problem~\eqref{eq:3}. 
\begin{mdframed} 
  Given a convex set $Q ⊆ ℝ^m$,  and a set ${\cal C} ⊆ ℤ^m$ with  $\|\C\|_∞ ≤ (m ⋅Δ)^{O(m)}$. Decide the validity of the statement 
 \begin{equation*}
   ∀ b ∈Q ∩ ℤ^m \,  ∃ c ∈{\cal C}：   c≤ b. 
 \end{equation*}
\end{mdframed}

\section{Diagonal Frobenius Number}
\label{sec:polyh-frob-probl}

A central element of our sequence of reductions is Lemma~\ref{lem:2} which is based on proximity between integer and fractional optimal solutions. We conclude this paper with a structural result concerning the following variant of the  forall-exist statement~\eqref{eq:7} in which the convex set $Q$ is the entire cone
\begin{displaymath}
  Q = \cone(W). 
\end{displaymath}
Our technique can be used to describe a subset of $\cone(W) ∩Λ(W)$  in which there is no counterexample. In other words,  every point belongs to the set $\intcone(W)$.

Similar results of this flavor have appeared in the recent
literature. The authors of~\cite{cslovjecsek2024parameterized} present
a \emph{deep in the cone Lemma} which identifies this set as being
those lattice points that are far away from the boundary
of~$\cone(W)$. 
 The authors note that
such a result can also be deduced  from Aliev and
Henk~\cite{aliev2010feasibility} who provide a bound on their
so-called \emph{diagonal Frobenius number}, which is the number $t^*$ below.  Given
a matrix $W \in \Z^{m \times n}$ such that $\cone(W)$ is pointed, find
the smallest natural number $t^*$ such that for all
$z \in \{Wx \colon x \geq t \mathbf{1} \} \cap Λ(W)$,
$z ∈ \intcone(W)$. Recall that a cone is \emph{pointed} if it does not contain a line, see, e.g.~\cite{schrijver1998theory}.  
The upper bound on the diagonal Frobenius number given by Aliev and Henk~\cite{aliev2010feasibility} is as follows.  
\begin{theorem}[\cite{aliev2010feasibility}]
  \label{DerSchoeneHenk}
    Let $W \in \Z^{m\times n}$ such that $Λ(W) = \Z^m$ with $\cone(W)$ pointed. Then the diagonal Frobenius number of $W$ is at most 
    \begin{align*}
        t^* = \frac{(n-m) \sqrt{n}}{2} \sqrt{\det(WW^{\top})}. 
    \end{align*}    
  \end{theorem}
  The goal of this section is to provide a simple proof bounding the diagonal Frobenius number in terms of the parameters $m$ (number of rows of $W$)  and $Δ$ (largest absolute value of a component of $W$). To explain the differences of our bound and the bound in Theorem~\ref{DerSchoeneHenk} in this setting, we first express the bound above in these parameters.

  Each component of  $W^T W$ is bounded by $n⋅ Δ$ in absolute value. Recall that $n ≤ (2Δ+1)^{m}$. The Hadamard bound implies
  \begin{align*}
     \sqrt{\det(WW^{T})} ≤  m^{O(m)} Δ^{O(m^2)}. 
\end{align*}
Thus, the upper bound~\cite{aliev2010feasibility} on the diagonal Frobenius number is
\begin{displaymath}
  t^* =  Δ^{O(m^2)}. 
\end{displaymath}
We will show below $t^* = (m Δ)^{O(m)}$. 

\begin{theorem}\label{thm:diag}
     Let $W \in \Z^{m \times n}$ and  $\cone(W)$ be pointed.  Then 
    \begin{align*}
    t^* \leq m ⋅(2m Δ +1)^m.
    \end{align*}        
\end{theorem}

\begin{proof}
  Let $b \in Λ(W)$ with $b = W λ$, $λ ∈ ℝ^n_{≥0}$ such that $λ \geq \mathbf{1} t$ and $t = m ⋅(2m Δ +1)^m $. To show  is  $b ∈ \intcone(W)$.  
  Let $b' = t ⋅ W \mathbf{1} ∈ Λ(W)$. We now consider the integer program
  \begin{equation}
    \label{eq:18}
    \max\left\{ \mathbf{0}^T (x^+,x^-) ： W x^+ - W x^- = b - b', \, (x^+,x^-) ≥ \bm{0}, \, (x^+,x^-) ∈ℤ^{2n}\right\}. 
  \end{equation}
  This integer program is feasible, since $b-b' ∈ Λ(W)$. Since $b-b' ∈ \cone(W)$, there exists an LP-optimal fractional solution $(y^+,y^-) ≥0$ such that $y^- = \bm{0}$. The proximity Theorem~\ref{thm:prox} implies that there exists an integer solution   $(z^+,z^-) ∈ℤ^{2n}$ such that $\|z^-\|_1 ≤  m ⋅(2m Δ +1)^m$. Notice that
  \begin{equation*}
    \label{eq:19}
    W (z^+-z^- + t ⋅ \mathbf{1}) = b \quad \text{ and } z^+-z^- + t ⋅ \mathbf{1} ∈ ℤ_{≥0}^n. 
  \end{equation*}
Hence, $b ∈ \intcone(W)$. 
\end{proof}

\bibliographystyle{plain}
\bibliography{ref,bibliography}

\end {document}

%% file: ipestyle.tex
\usetikzlibrary{arrows.meta,patterns}
\usetikzlibrary{ipe} 

\tikzstyle{ipe stylesheet} = [
  ipe import,
  even odd rule,
  line join=round,
  line cap=butt,
  ipe pen normal/.style={line width=0.4},
  ipe pen heavier/.style={line width=0.8},
  ipe pen fat/.style={line width=1.2},
  ipe pen ultrafat/.style={line width=2},
  ipe pen normal,
  ipe mark normal/.style={ipe mark scale=3},
  ipe mark large/.style={ipe mark scale=5},
  ipe mark small/.style={ipe mark scale=2},
  ipe mark tiny/.style={ipe mark scale=1.1},
  ipe mark normal,
  /pgf/arrow keys/.cd,
  ipe arrow normal/.style={scale=7},
  ipe arrow large/.style={scale=10},
  ipe arrow small/.style={scale=5},
  ipe arrow tiny/.style={scale=3},
  ipe arrow normal,
  /tikz/.cd,
  ipe arrows, 
  <->/.tip = ipe normal,
  ipe dash normal/.style={dash pattern=},
  ipe dash dotted/.style={dash pattern=on 1bp off 3bp},
  ipe dash dashed/.style={dash pattern=on 4bp off 4bp},
  ipe dash dash dotted/.style={dash pattern=on 4bp off 2bp on 1bp off 2bp},
  ipe dash dash dot dotted/.style={dash pattern=on 4bp off 2bp on 1bp off 2bp on 1bp off 2bp},
  ipe dash normal,
  ipe node/.append style={font=\normalsize},
  ipe stretch normal/.style={ipe node stretch=1},
  ipe stretch normal,
  ipe opacity 10/.style={opacity=0.1},
  ipe opacity 30/.style={opacity=0.3},
  ipe opacity 50/.style={opacity=0.5},
  ipe opacity 75/.style={opacity=0.75},
  ipe opacity opaque/.style={opacity=1},
  ipe opacity opaque,
]

\definecolor{red}{rgb}{1,0,0}
\definecolor{blue}{rgb}{0,0,1}
\definecolor{green}{rgb}{0,1,0}
\definecolor{yellow}{rgb}{1,1,0}
\definecolor{orange}{rgb}{1,0.647,0}
\definecolor{gold}{rgb}{1,0.843,0}
\definecolor{purple}{rgb}{0.627,0.125,0.941}
\definecolor{gray}{rgb}{0.745,0.745,0.745}
\definecolor{brown}{rgb}{0.647,0.165,0.165}
\definecolor{navy}{rgb}{0,0,0.502}
\definecolor{pink}{rgb}{1,0.753,0.796}
\definecolor{seagreen}{rgb}{0.18,0.545,0.341}
\definecolor{turquoise}{rgb}{0.251,0.878,0.816}
\definecolor{violet}{rgb}{0.933,0.51,0.933}
\definecolor{darkblue}{rgb}{0,0,0.545}
\definecolor{darkcyan}{rgb}{0,0.545,0.545}
\definecolor{darkgray}{rgb}{0.663,0.663,0.663}
\definecolor{darkgreen}{rgb}{0,0.392,0}
\definecolor{darkmagenta}{rgb}{0.545,0,0.545}
\definecolor{darkorange}{rgb}{1,0.549,0}
\definecolor{darkred}{rgb}{0.545,0,0}
\definecolor{lightblue}{rgb}{0.678,0.847,0.902}
\definecolor{lightcyan}{rgb}{0.878,1,1}
\definecolor{lightgray}{rgb}{0.827,0.827,0.827}
\definecolor{lightgreen}{rgb}{0.565,0.933,0.565}
\definecolor{lightyellow}{rgb}{1,1,0.878}
\definecolor{black}{rgb}{0,0,0}
\definecolor{white}{rgb}{1,1,1}

%% file: arrangement-new.tex
\tikzstyle{ipe stylesheet} = [
  ipe import,
  even odd rule,
  line join=round,
  line cap=butt,
  ipe pen normal/.style={line width=0.4},
  ipe pen heavier/.style={line width=0.8},
  ipe pen fat/.style={line width=1.2},
  ipe pen ultrafat/.style={line width=2},
  ipe pen normal,
  ipe mark normal/.style={ipe mark scale=3},
  ipe mark large/.style={ipe mark scale=5},
  ipe mark small/.style={ipe mark scale=2},
  ipe mark tiny/.style={ipe mark scale=1.1},
  ipe mark normal,
  /pgf/arrow keys/.cd,
  ipe arrow normal/.style={scale=7},
  ipe arrow large/.style={scale=10},
  ipe arrow small/.style={scale=5},
  ipe arrow tiny/.style={scale=3},
  ipe arrow normal,
  /tikz/.cd,
  ipe arrows, 
  <->/.tip = ipe normal,
  ipe dash normal/.style={dash pattern=},
  ipe dash dotted/.style={dash pattern=on 1bp off 3bp},
  ipe dash dashed/.style={dash pattern=on 4bp off 4bp},
  ipe dash dash dotted/.style={dash pattern=on 4bp off 2bp on 1bp off 2bp},
  ipe dash dash dot dotted/.style={dash pattern=on 4bp off 2bp on 1bp off 2bp on 1bp off 2bp},
  ipe dash normal,
  ipe node/.append style={font=\normalsize},
  ipe stretch normal/.style={ipe node stretch=1},
  ipe stretch normal,
  ipe opacity 10/.style={opacity=0.1},
  ipe opacity 30/.style={opacity=0.3},
  ipe opacity 50/.style={opacity=0.5},
  ipe opacity 75/.style={opacity=0.75},
  ipe opacity opaque/.style={opacity=1},
  ipe opacity opaque,
]
\definecolor{red}{rgb}{1,0,0}
\definecolor{blue}{rgb}{0,0,1}
\definecolor{green}{rgb}{0,1,0}
\definecolor{yellow}{rgb}{1,1,0}
\definecolor{orange}{rgb}{1,0.647,0}
\definecolor{gold}{rgb}{1,0.843,0}
\definecolor{purple}{rgb}{0.627,0.125,0.941}
\definecolor{gray}{rgb}{0.745,0.745,0.745}
\definecolor{brown}{rgb}{0.647,0.165,0.165}
\definecolor{navy}{rgb}{0,0,0.502}
\definecolor{pink}{rgb}{1,0.753,0.796}
\definecolor{seagreen}{rgb}{0.18,0.545,0.341}
\definecolor{turquoise}{rgb}{0.251,0.878,0.816}
\definecolor{violet}{rgb}{0.933,0.51,0.933}
\definecolor{darkblue}{rgb}{0,0,0.545}
\definecolor{darkcyan}{rgb}{0,0.545,0.545}
\definecolor{darkgray}{rgb}{0.663,0.663,0.663}
\definecolor{darkgreen}{rgb}{0,0.392,0}
\definecolor{darkmagenta}{rgb}{0.545,0,0.545}
\definecolor{darkorange}{rgb}{1,0.549,0}
\definecolor{darkred}{rgb}{0.545,0,0}
\definecolor{lightblue}{rgb}{0.678,0.847,0.902}
\definecolor{lightcyan}{rgb}{0.878,1,1}
\definecolor{lightgray}{rgb}{0.827,0.827,0.827}
\definecolor{lightgreen}{rgb}{0.565,0.933,0.565}
\definecolor{lightyellow}{rgb}{1,1,0.878}
\definecolor{black}{rgb}{0,0,0}
\definecolor{white}{rgb}{1,1,1}
\begin{tikzpicture}[ipe stylesheet]
  \filldraw[lightgray, ipe pen fat]
    (336, 800) rectangle (464, 576);
  \filldraw[lightgray, ipe pen fat]
    (256, 800) rectangle (464, 640);
  \filldraw[lightgray, ipe pen fat]
    (192, 800) rectangle (464, 720);
  \draw[darkgray, ipe pen fat]
    (192, 800)
     -- (192, 544)
     -- (192, 544);
  \draw[darkgray, ipe pen fat]
    (256, 800)
     -- (256, 544)
     -- (256, 544)
     -- (256, 544)
     -- (256, 544);
  \draw[darkgray, ipe pen fat]
    (336, 800)
     -- (336, 544)
     -- (336, 544)
     -- (336, 544)
     -- (336, 544)
     -- (336, 544)
     -- (336, 544);
  \draw[darkgray, ipe pen fat]
    (160, 768)
     -- (464, 768);
  \draw[gray, ipe pen fat]
    (160, 720)
     -- (464, 720);
  \draw[darkgray, ipe pen fat]
    (416, 800)
     -- (416, 544);
  \draw[darkgray, ipe pen fat]
    (160, 640)
     -- (464, 640);
  \draw[darkgray, ipe pen fat]
    (160, 576)
     -- (464, 576);
  \pic[darkgray]
     at (192, 720) {ipe disk};
  \pic[darkgray]
     at (256, 640) {ipe disk};
  \pic[darkgray]
     at (336, 576) {ipe disk};
  \pic
     at (192, 720) {ipe disk};
  \pic
     at (256, 640) {ipe disk};
  \pic
     at (336, 576) {ipe disk};
  \draw[ipe pen fat]
    (192, 800)
     -- (192, 720)
     -- (256, 720)
     -- (256, 640)
     -- (336, 640)
     -- (336, 576)
     -- (464, 576);
  \pic[red]
     at (224, 672) {ipe disk};
  \node[ipe node]
     at (224, 656) {$b \in Q \cap {Z}^m$};
  \draw[blue, ipe pen ultrafat]
    (344, 602.6667)
     .. controls (341.3333, 565.3333) and (250.6667, 602.6667) .. (208, 658.6667)
     .. controls (165.3333, 714.6667) and (170.6667, 789.3333) .. (216, 770.6667)
     .. controls (261.3333, 752) and (346.6667, 640) .. cycle;
  \node[ipe node]
     at (192, 704) {$c_1$};
  \node[ipe node]
     at (336, 560) {$c_3$};
  \node[ipe node]
     at (256, 624) {$c_2$};
\end{tikzpicture}